\documentclass[10pt]{article}

\usepackage{amsmath,amscd,amsfonts,amsthm,verbatim, xypic, amssymb, graphicx}
\input{xy}

\newcommand{\shrinkmargins}[1]{
  \addtolength{\textheight}{#1\topmargin}
  \addtolength{\textheight}{#1\topmargin}
  \addtolength{\textwidth}{#1\oddsidemargin}
  \addtolength{\textwidth}{#1\evensidemargin}
  \addtolength{\topmargin}{-#1\topmargin}
  \addtolength{\oddsidemargin}{-#1\oddsidemargin}
  \addtolength{\evensidemargin}{-#1\evensidemargin}
  }

\shrinkmargins{.7}

\DeclareMathOperator{\Conf}{Conf}
\DeclareMathOperator{\Grav}{Grav}

\DeclareMathOperator{\T}{T}
\DeclareMathOperator{\TTH}{TH}
\DeclareMathOperator{\Aff}{Aff}

\newcommand{\field}[1]{\mathbb{#1}}

\newcommand{\Z}{\field{Z}}

\newcommand{\R}{\field{R}}
\newcommand{\C}{\field{C}}
\renewcommand{\P}{\field{P}}

\newcommand{\DD} {\mathcal{D}}

\newcommand{\sTH} {\mathcal{TH}}

\newcommand{\im}{{\rm im \,}}

\newcommand{\MM}{\mathcal{M}}

\newcommand{\beq}{\begin{displaymath}}
\newcommand{\eeq}{\end{displaymath}}
\newcommand{\beqn}{\begin{equation}}
\newcommand{\eeqn}{\end{equation}}

\newcommand{\ibar}{\overline{i}}
\newcommand{\pbar}{\overline{p}}
\newcommand{\qbar}{\overline{q}}

\theoremstyle{plain}
\newtheorem{thm}{Theorem}[section]
\newtheorem{prop}[thm]{Proposition}
\newtheorem{cor}[thm]{Corollary}

\theoremstyle{definition}
\newtheorem{defn}[thm]{Definition}

\theoremstyle{remark}
\newtheorem{rem}[thm]{Remark}

\title{Operads of moduli spaces of points in $\C^d$}
\author{Craig Westerland}
\begin{document}
\bibliographystyle{amsalpha}

\maketitle

\begin{abstract}

We compute the structure of the homology of an operad built from the spaces $\TTH_{d, n}$ of configurations of points in $\C^d$, modulo translation and homothety.  We find that it is a mild generalization of Getzler's gravity operad, which occurs in dimension $d=1$.

\end{abstract}


\section{Introduction}

In a wealth of papers, e.g., \cite{getzler_2d,getzler_0, ginzburg_kapranov, getzler_kapranov, voronov, kimura_stasheff_voronov}, a number of connections between moduli spaces of curves and operads have been firmly established.  In this note, we explore an operad built out of moduli spaces of points in higher-dimensional objects.

In \cite{chen_gibney_krashen}, Chen, Gibney, and Krashen study a variety $\TTH_{d, n}$ of configurations of $n$ points in affine $d$-space modulo the action of the affine group, and define a compactification $\T_{d, n}$ of this variety.  In dimension $d=1$, these varieties return the familiar moduli spaces of points in $\P^1$: $\TTH_{1, n} = \MM_{0, n+1}$ and $\T_{1, n}$ is the Deligne-Mumford compactification $\overline{\MM}_{0, n+1}$.

Just as in dimension 1, $\TTH_{d,n}$ and $\T_{d, n}$ (or, for our purposes, their complex points) give rise to operads.  In the case of $\T_{d, n}$, as for $\overline{\MM}_{0, n+1}$, this structure arises via grafting of trees of projective spaces (as in a free operad).  The operadic structure on $\TTH_{d, n}$ may be derived from this via a form of transfer, though this is not quite the approach we take here.  Write $H_*(\TTH_{d})$ for the operad whose $n^{\rm th}$ term is $\Sigma H_*(\TTH_{d, n})$ (here $\Sigma$ indicates a shift of degree by 1).

\begin{defn}

Let $\Grav_d$ be the operad of graded $\Z$-modules generated by $k$-ary operations $\{a_1, \dots, a_k\} \in \Grav_d(k)$ of dimension $2d-1$, and $c \in \Grav_d(1)$, of dimension $-2$, subject to the relations
$$\{\{a_1, \dots, a_k\}, b_1, \dots, b_l \} = \sum_{1 \leq i < j \leq k} (-1)^{\epsilon(i, j)} \{ \{a_i, a_j \}, a_1, \dots, \hat{a}_i, \dots, \hat{a}_j, \dots, a_k, b_1, \dots, b_l \},$$
$$\begin{array}{ccc} c^d = 0, & {\rm and} & c \cdot \{a_1, \dots, a_k\} = \{a_1, \dots, c \cdot  a_i, \dots, a_k \}, \forall i. \end{array}$$

\end{defn}

If $d=1$, this is precisely the gravity operad introduced by Getzler in \cite{getzler_2d, getzler_0}, where it was shown to be isomorphic to the operad $\Sigma H_*(\MM_{0, n+1})$.  It is the purpose of this note to extend this result to the higher-dimensional setting:

\begin{thm} \label{main_thm}

There is an isomorphism of operads $H_*(\TTH_{d}) \cong \Grav_d$ in ``arity" $n>1$.

\end{thm}

We expect this computation to be useful in determining the structure of the homology of the operad $\T_d$.  One concrete application of this result is as follows (derived from Theorem \ref{main_thm} and \cite{goperads}):

\begin{cor}

Let $X = \Omega^{2d} Y$ for an $S^1$-space $Y$ (more generally, let $X$ be an algebra over the $(2d)$-dimensional framed little disks operad).  Then the shifted equivariant homology $\Sigma H_*^{S^1}(X)$ is an algebra over the suboperad $(\Grav_d)_{>1}$ of arity $>1$.

\end{cor}

It is a pleasure to thank Michael Ching and Danny Krashen for helpful conversations about this material.  The author was partially supported by NSF grant DMS-0705428 and ARC grant DP1095831.


\section{The cohomology of $\TTH_{d, n}$}

Recall that the ordered configuration space of $n$ points in $\C^d$, $\Conf_n(\C^d)$, is the space
$$\Conf_n(\C^d) = \{ (x_1, \dots, x_n) \; | \; x_i \neq x_j \mbox{ if $i \neq j$} \} \subseteq (\C^d)^{\times n}.$$
This space is acted upon (component-wise) by the affine group $\Aff(\C^d) \cong \C^{\times} \rtimes \C^d$.  If $n>1$, the action is free.

\begin{defn}

For $n>1$, $\TTH_{d, n} := \Conf_n(\C^d) / \Aff(\C^d)$.

\end{defn}

The affine group is homotopy equivalent to its subgroup $S^1 = U(1)$, so there is a homotopy equivalence $\TTH_{d, n} \simeq \Conf_n(\C^d) / S^1$.

Define $p = p_{12}: \Conf_n(\C^d) \to \Conf_2(\C^d)$ by $p(x_1, \dots, x_n) = (x_1, x_2)$.  In general, write $p_{ij}(x_1, \dots, x_n) = (x_i, x_j)$.  This is a fibration, and is equivariant for the $S^1$-action.  Therefore, there is a commutative diagram of fibrations
$$\xymatrix{
\Conf_{n-2}(\C^d \setminus \{a, b\}) \ar[r]^-i \ar[d]_-= & \Conf_n(\C^d) \ar[r]^-p \ar[d]^-{q} & \Conf_2(\C^d) \ar[d]^-{\qbar} \\
\Conf_{n-2}(\C^d \setminus \{a, b\}) \ar[r]_-{\ibar} & \Conf_n(\C^d) / S^1 \ar[r]_-\pbar & \Conf_2(\C^d) / S^1
}$$
where $a, b$ are fixed, distinct points in $\C^d$.

Now, $\Conf_2(\C^d)$ is homotopy equivalent to $S^{2d-1}$, so for degree reasons the Serre spectral sequence for $p$ collapses at $E_2$.  This allowed \cite{clm} to prove that there is a ring isomorphism
$$H^*(\Conf_n(\C^d)) = \Lambda[x_{ij} \; | \; 1 \leq i \neq j \leq n] / (x_{ij} = x_{ji},\; x_{ij} x_{jk} + x_{jk} x_{ki} + x_{ki} x_{ij})$$
where $x_{ij}$ is the pullback under $p_{ij}^*$ of the generator of $H^{2d-1}(\Conf_2(\C^d)) = \Z$.  (See also \cite{getzler_jones}).

Consequently, one can identify $H^*(\Conf_{n-2}(\C^d \setminus \{a, b\}))$ as a quotient of $H^*(\Conf_n(\C^d))$:
\begin{eqnarray*}
H^*(\Conf_{n-2}(\C^d \setminus \{a, b\})) & \cong &  H^*(\Conf_n(\C^d))/(x_{12})\\
 & = & \Lambda[x_{ij} \; | \; 1 \leq i \neq j \leq n] / (x_{ij} = x_{ji},\; x_{ij} x_{jk} + x_{jk} x_{ki} + x_{ki} x_{ij}, \; x_{12})
\end{eqnarray*}

Now, since $\Conf_2(\C^d)$ is $S^1$-equivariantly homotopy equivalent to $S^{2d-1}$, $\Conf_2(\C^d) / S^1\simeq \C P^{d-1}$.  Therefore the Serre spectral sequence for $\pbar$ is of the form
$$E_2^{*, *} = H^*(\C P^{d-1}) \otimes H^*(\Conf_{n-2}(\C^d \setminus \{a, b\})) \implies H^*(\TTH_{d, n})$$
Again, the spectral sequence collapses because all differentials are determined on the fibre, and there are no possible targets for generators of the cohomology of the fibre for degree reasons.  We conclude:

\begin{prop} \label{comp_prop}

There is a ring isomorphism
\begin{eqnarray*} H^*(\TTH_{d, n}) & = &  H^*(\C P^{d-1}) \otimes H^*(\Conf_{n-2}(\C^d \setminus \{a, b\}))\\
 & = & \Z[c]/(c^d) \otimes  \Lambda[x_{ij} \; | \; 1 \leq i \neq j \leq n] / (x_{ij} = x_{ji},\; x_{ij} x_{jk} + x_{jk} x_{ki} + x_{ki} x_{ij}, \; x_{12})
\end{eqnarray*}

\end{prop}

It is worth remarking that in dimension $d=1$, this is a reflection of the wholly unsurprising fact that there is a homeomorphism $\MM_{0, n+1} \cong \Conf_{n-2}(\C \setminus \{ 0, 1 \})$.


\section{The action of $U(d)$}

Notice that there is an action of $U(d)$ on $\Conf_n(\C^d)$, of which the $S^1 = U(1)$-action is but a part.  The Pontrjagin ring of $U(d)$ is
$$H_*(U(d)) = \Lambda[\Delta_1, \dots, \Delta_d]$$
where $\Delta_k$ is a generator of dimension $2k-1$, obtained iteratively from fibrations over odd-dimensional spheres (see, e.g., \cite{mimura_toda, salvatore_wahl}).  These classes induce natural maps
$$\Delta_k :H_*(\Conf_n(\C^d)) \to H_{*+(2k-1)}(\Conf_n(\C^d))$$
via the group action.  

\begin{prop} \label{U_d_prop}

$H_*(\Conf_n(\C^d))$ is a free $\Lambda[\Delta_d]$-module over $H_*(\Conf_{n-2}(\C^d \setminus \{a, b\}))$.

\end{prop}

\begin{proof}

We use the dual action, in cohomology.  That is, $H_*(U(d)$ acts on $H^*(\Conf_n(\C^d))$ via dual maps $\Delta_k^*$ which \emph{decrease} degree by $2k-1$.  Because each $\Delta_k$ is primitive, $\Delta_k^*$ is a derivation.  It is easy to see for degree reasons that the action of $\Delta_k^*$ on $H^*(\Conf_n(\C^d))$ is null except when $k=d$, and there,
$$\Delta_d^*(x_{ij}) = 1, \; \; \forall ij.$$

If we define $y_{ij} := x_{ij} - x_{12}$, then $H^*(\Conf_n(\C^d))$ is generated multiplicatively by $y_{ij}$, $ij \neq 12$ along with $x_{12}$.  Write $Y$ for the subalgebra generated by $\{ y_{ij} \; | \; ij \neq 12 \}$.  By the computations above, $i^*$ carries $Y$ isomorphically onto $H^*(\Conf_{n-2}(\C^d \setminus \{a, b\}))$.

Note that
$$H^*(\Conf_n(\C^d)) = Y \oplus Y\cdot x_{12}$$
That is, $H^*(\Conf_n(\C^d))$ is a free $\Lambda[x_{12}]$-module, generated by $Y$.  Clearly $\Delta_d^*(y_{ij}) = 0$, and since $\Delta_d^*$ is a derivation, this implies that $Y \subseteq \ker \Delta_d^*$.  For a general element $y + y' x_{12}$, we see that
$$\Delta_d^*(y + y' x_{12}) = 0 + 0 \cdot x_{12} + y' \Delta_d^*(x_{12}) = y'$$
so in fact, $Y = \ker \Delta_d^*$.  We conclude that as a $\Lambda[\Delta_d^*]$-module, $H^*(\Conf_n(\C^d))$ is a free over $Y\cdot x_{12}$.  Dually, $H_*(\Conf_n(\C^d))$ is therefore a free $\Lambda[\Delta_d]$-module over $Y^* = i_*(H_*(\Conf_{n-2}(\C^d \setminus \{a, b\})))$.

\end{proof}

\begin{rem}

This implies that the subspace $\ker\Delta_d = \im\Delta_d$ is isomorphic to the shifted copy 
$$\ker\Delta_d \cong \Sigma^{2d-1} H_*(\Conf_{n-2}(\C^d \setminus \{a, b\})).$$

\end{rem}


\section{$\TTH_{d}$ as an operad}

\begin{prop}

For each $d>0$, there is an operad $\sTH_d$ in the category of $S$-modules whose $n^{\rm th}$ term $\sTH_{d}(n)$ is weakly equivalent to the (shifted) suspension spectrum $\Sigma \Sigma^{\infty} (\TTH_{d, n})_+$ for $n>1$.

\end{prop}

The category of $S$-modules, introduced in \cite{ekmm} is a rigidification of the stable homotopy category of spectra to admit a symmetric monoidal smash product.  For those with little background or patience for the stable homotopy category, this proposition has the immediate (and down-to-earth) consequence:

\begin{cor} \label{h_*_cor}

The collection $H_*(\TTH_{d})(n) := \Sigma H_*(\TTH_{d, n})$, $n>1$, form a (non-unital) operad in the category of graded abelian groups.

\end{cor}

We note that the shift by 1 is important; it accounts for a degree shifting $S^1$-transfer map inherent in this construction.  On $\sTH_d$, this transfer exists as an actual map between the spectra forming the operad.  For $H_*(\TTH_{d})$, it comes from a homological transfer map: for an $S^1$-bundle $E \to B$, the transfer sends an element of $H_q(B)$ to the $(q+1)$-dimensional cycle lying over it in $H_{q+1}(E)$.

\begin{proof}

Let $\DD_{2d}$ denote the operad of $2d$-dimensional little disks, after \cite{may}.  In \cite{salvatore_wahl}, this was shown to be an $SO(2d)$-operad.  Consider the group homomorphism $S^1 \to SO(2d)$ where $z \in S^1$ acts on $\C^d = \R^{2d}$ by $z \cdot (z_1, \dots, z_d) = (z \cdot z_1, \dots, z \cdot z_d)$.  By restriction, this makes $\DD_{2d}$ into an $S^1$-operad.  Using the machinery of \cite{goperads}, we define $\sTH_d$ as the \emph{homotopy fixed point operad} $\sTH_d := \DD_{2d}^{h S^1}$.

Now $\DD_{2d}(n)$ is $S^1$-equivariantly homotopy equivalent to $\Conf_n(\C^d)$.  Moreover, since the action of $S^1$ on the latter space is free ($n>1$), and its quotient $\TTH_{d, n}$ is equivalent to a finite CW complex, $\DD_{2d}(n)$ is $S^1$-equivariantly finitely dominated.   Thus by Theorem D of \cite{klein_dual}, the norm map gives a homotopy equivalence
$$\sTH_d(n) = \DD_{2d}(n)^{h S^1} \simeq \Sigma \Sigma^{\infty} (\DD_{2d}(n)_{h S^1})_+ \simeq \Sigma \Sigma^{\infty} (\TTH_{d, n})_+$$

\end{proof}

Although this result does not apply to the unary part of the homotopy fixed point operad (i.e., the Spanier-Whitehead dual $\DD_{2d}^{hS^1}(1) = F(BS^1_+, S^0)$), it will play a role in the section below in studying the interaction of the Chern class with the rest of the operad.

 A low-technology proof of Corollary \ref{h_*_cor} is given in section 3.2 of \cite{goperads}.


\section{The proof of Theorem \ref{main_thm}}

Since $\DD_{2d}$ is an $SO(2d)$-operad (and hence $U(d)$-operad), $H_*(\DD_{2d})$ is a $H_*(U(d))$-operad.  Moreover, the primitivity of $\Delta_d$ implies that it is a derivation for the \emph{operad composition} on $H_*(\DD_{2d}(n)) = H_*(\Conf_n(\C^d))$.  That is, the operad compositions $\circ_i$ satisfy
$$\Delta_k(a \circ_i b) = (\Delta_k a) \circ_i b + (-1)^{|a|} a \circ_i (\Delta_k b).$$
See \cite{salvatore_wahl}.  A consequence of this fact is therefore that $\ker \Delta_d \subseteq H_*(\DD_{2d}) = H_*(\Conf_*(\C^d))$ is a suboperad.  One can now copy the proof of Theorem 4.5 of \cite{getzler_2d} to get

\begin{thm} \label{ker_thm}

The operad $\ker \Delta_d$ is generated by operations $\{a_1, \dots, a_k\}$ of ``arity" $k$ (the $a_i$ are dummy variables) of dimension $2d-1$, subject to relations
$$\{\{a_1, \dots, a_k\}. b_1, \dots, b_l \} = \sum_{1 \leq i < j \leq k} (-1)^{\epsilon(i, j)} \{ \{a_i, a_j \}, a_1, \dots, \hat{a}_i, \dots, \hat{a}_j, \dots, a_k, b_1, \dots, b_l \}$$
where $\epsilon(i, j) = (|a_1| + \dots + |a_{i-1}|)|a_i| + (|a_1| + \dots + |a_{j-1}|)|a_j| + |a_i||a_j|$.

\end{thm}

This almost proves Theorem \ref{main_thm}; what remains is to identify 
$$H_*(\TTH_{d, n}) = \Z[c]/(c^d) \otimes \ker \Delta_d,$$
and to show that the operad structure behaves as indicated.

We note that there is a natural map
$$\phi: \TTH_{d, n} \to BS^1$$
that classifies the principal $S^1$-bundle $\Conf_n(\C^d) \to \TTH_{d, n}$.  This makes $H_*(\TTH_{d, n})$ into an $H^*(BS^1)$-module by
$$\alpha \cdot x := \phi^*(\alpha) \cap x$$
In this setting, the first Chern class $c\in H^2(BS^1)$ acts as if it were dimension $-2$.  By Proposition \ref{comp_prop}, $c^d$ acts as $0$, making $H_*(\TTH_{d})(n) = \Sigma H_*(\TTH_{d, n})$ a free $\Z[c]/c^d$-module over 
$$\Sigma^{2d-1} H_*(\Conf_{n-2}(\C^d \setminus \{ a, b \})) = \ker\Delta_d.$$  
Here we have shifted $H_*(\Conf_{n-2}(\C^d \setminus \{ a, b \}))$ up in dimension by $2(d-1)$ to make up for the fact that $\Z[c]/c^d$ acts by decreasing degree.  

By Proposition \ref{U_d_prop} and Theorem \ref{ker_thm}, then $H_*(\TTH_{d, n}) =\Grav_d(n)$; all relations are verified except
$$c \cdot \{a_1, \dots, a_k\} = \{a_1, \dots, c \cdot  a_i, \dots, a_k \}$$
This is, however, automatic in the continuous cohomology of the operad $\DD_{2d}^{hS^1}$, as seen in \cite{goperads}.  More concretely, one can proceed as follows.  The composition in $\DD_{2d}$ is that of an $S^1$-operad: each $\circ_i$ map
$$\xymatrix@1{
\DD_{2d}(k) \times \DD_{2d}(l) \ar[r]^-{\circ_i} & \DD_{2d}(k+l-1)
}$$
is equivariant (where $S^1$ acts diagonally on the left side).  Take $l=1$; then $\DD_{2d}(1)$ is contractible, so, oddly, this diagram commutes up to $S^1$-equivariant homotopy:
$$\xymatrix@1{
\DD_{2d}(k) \times \DD_{2d}(1) \ar[d]_-{T}^-{\simeq} \ar[r]^-{\circ_i}_-{\simeq} & \DD_{2d}(k) \\
\DD_{2d}(1) \times \DD_{2d}(k) \ar[ur]^-{\simeq}_-{\circ_1} & 
}$$
($T$ is the map that switches factors).  Quotienting by $S^1$, we have a homotopy commutative diagram
$$\xymatrix@1{
\TTH_{d, k+1} \times BS^1 \ar[d]_-{T}^-{\simeq} & \DD_{2d}(k) \times_{S^1} \DD_{2d}(1) \ar[l] \ar[d]_-{T}^-{\simeq} \ar[r]^-{\circ_i}_-{\simeq} & \DD_{2d}(k) / S^1 \\
BS^1 \times \TTH_{d, k+1} & \DD_{2d}(1) \times_{S^1} \DD_{2d}(k) \ar[l] \ar[ur]^-{\simeq}_-{\circ_1} & 
}$$
Apply (co)homology; the relations follow by comparing the passage along the top and bottom rows.

\bibliography{biblio}

\providecommand{\bysame}{\leavevmode\hbox to3em{\hrulefill}\thinspace}
\providecommand{\MR}{\relax\ifhmode\unskip\space\fi MR }
\providecommand{\MRhref}[2]{%
  \href{http://www.ams.org/mathscinet-getitem?mr=#1}{#2}
}
\providecommand{\href}[2]{#2}
\begin{thebibliography}{EKMM97}

\bibitem[CGK09]{chen_gibney_krashen}
L.~Chen, A.~Gibney, and D.~Krashen, \emph{Pointed trees of projective spaces},
  J. Algebraic Geom. \textbf{18} (2009), no.~3, 477--509.

\bibitem[CLM76]{clm}
F.~R. Cohen, T.~J. Lada, and J.~P. May, \emph{{The Homology of Iterated Loop
  Spaces}}, Lecture Notes in Mathematics 533, Springer Verlag, Berlin, 1976.

\bibitem[EKMM97]{ekmm}
A.~D. Elmendorf, I.~Kriz, M.~A. Mandell, and J.~P. May, \emph{Rings, modules,
  and algebras in stable homotopy theory}, Mathematical Surveys and Monographs,
  vol.~47, American Mathematical Society, Providence, RI, 1997, With an
  appendix by M. Cole.

\bibitem[Get94]{getzler_2d}
E.~Getzler, \emph{Two-dimensional topological gravity and equivariant
  cohomology}, Comm. Math. Phys. \textbf{163} (1994), no.~3, 473--489.

\bibitem[Get95]{getzler_0}
\bysame, \emph{Operads and moduli spaces of genus {$0$} {R}iemann surfaces},
  The moduli space of curves (Texel Island, 1994), Progr. Math., vol. 129,
  Birkh\"auser Boston, Boston, MA, 1995, pp.~199--230.

\bibitem[GJ94]{getzler_jones}
E.~Getzler and J.~D.~S. Jones, \emph{Operads, homotopy algebra and iterated
  integrals for double loop spaces}, preprint: hep-th/9403055 (1994).

\bibitem[GK94]{ginzburg_kapranov}
V.~Ginzburg and M.~Kapranov, \emph{Koszul duality for operads}, Duke Math. J.
  \textbf{76} (1994), no.~1, 203--272.

\bibitem[GK98]{getzler_kapranov}
E.~Getzler and M.~M. Kapranov, \emph{Modular operads}, Compositio Math.
  \textbf{110} (1998), no.~1, 65--126.

\bibitem[Kle01]{klein_dual}
J.~R. Klein, \emph{The dualizing spectrum of a topological group}, Math. Ann.
  \textbf{319} (2001), no.~3, 421--456.

\bibitem[KSV95]{kimura_stasheff_voronov}
T~Kimura, J~Stasheff, and A.~A. Voronov, \emph{On operad structures of moduli
  spaces and string theory}, Comm. Math. Phys. \textbf{171} (1995), no.~1,
  1--25.

\bibitem[May72]{may}
J.~P. May, \emph{{The Geometry of Iterated Loop Spaces}}, Lecture Notes in
  Mathematics 271, Springer Verlag, Berlin, 1972.

\bibitem[MT91]{mimura_toda}
M.~Mimura and H.~Toda, \emph{Topology of {L}ie groups. {I}, {II}}, Translations
  of Mathematical Monographs, vol.~91, American Mathematical Society,
  Providence, RI, 1991, Translated from the 1978 Japanese edition by the
  authors.

\bibitem[SW03]{salvatore_wahl}
P.~Salvatore and N.~Wahl, \emph{Framed discs operads and {B}atalin-{V}ilkovisky
  algebras}, Q. J. Math. \textbf{54} (2003), no.~2, 213--231.

\bibitem[Vor00]{voronov}
A.~A. Voronov, \emph{Homotopy {G}erstenhaber algebras}, Conference Moshe Flato
  1999, Math. Phys. Stud., 22, vol.~II, 2000, pp.~307--331.

\bibitem[Wes08]{goperads}
C.~Westerland, \emph{Equivariant operads, string topology, and {T}ate
  cohomology}, Math. Ann. \textbf{340} (2008), no.~1, 97--142.

\end{thebibliography}

\end{document}